\documentclass[12pt,twoside, a4paper]{amsart} 
\usepackage{hyperref}
\usepackage{amsmath,amsthm,amsopn,amssymb,a4wide,varioref}
\usepackage{color, bm, amscd, tikz-cd}
\newcommand{\cD}{{\mathcal D}}

\newcommand{\cH}{{\mathcal H}}

\newcommand{\cO}{{\mathcal O}}

\newcommand{\cV}{{\mathcal V}}
\newcommand{\cW}{{\mathcal W}}

\newcommand{\bT}{{\mathbb{T}}}

\newcommand{\bG}{{\mathbb G}}
\newcommand{\sbm}[1]{\left[\begin{smallmatrix} #1
		\end{smallmatrix}\right]}

\newcommand{\bD}{{\mathbb D}}
\newcommand{\bC}{{\mathbb C}}
\newcommand{\bB}{{\mathbb B}}
\newcommand{\bS}{{\mathbb S}}

\newtheorem{thm}{Theorem}[section]

\newtheorem{lemma}[thm]{Lemma}

\newtheorem{proposition}[thm]{Proposition}

\theoremstyle{definition}

\newtheorem{definition}[thm]{Definition}
\newtheorem{remark}[thm]{Remark}

\newtheorem{example}[thm]{Example}

\numberwithin{equation}{section}

\def\textmatrix#1&#2\\#3&#4\\{\bigl({#1 \atop #3}\ {#2 \atop #4}\bigr)}
\def\dispmatrix#1&#2\\#3&#4\\{\left({#1 \atop #3}\ {#2 \atop #4}\right)}
\numberwithin{equation}{section}

\def\textmatrix#1&#2\\#3&#4\\{\bigl({#1 \atop #3}\ {#2 \atop #4}\bigr)}
\def\dispmatrix#1&#2\\#3&#4\\{\left({#1 \atop #3}\ {#2 \atop #4}\right)}
\begin{document}
\title[Determining sets]{Determining sets for holomorphic functions on the symmetrized bidisk}
\author[Das]{B. Krishna Das}
\address[Das]{Department of Mathematics, Indian Institute of Technology Bombay, Powai, Mumbai, 400076, India}
\email{dasb@math.iitb.ac.in, bata436@gmail.com}

\author[Kumar]{P. Kumar}
\address[Kumar]{Department of Mathematics, Indian Institute of Science, Bengaluru-560012, India}
\email{poornendukumar@gmail.com}

\author[Sau]{H. Sau}
\address[Sau]{Indian Institute of Science Education and Research, Dr. Homi Bhabha Road, Pashan, Pune, Maharashtra 411008, India.}
\email{haripadasau215@gmail.com}

\begin{abstract}
	A subset $\cD$ of a domain $\Omega\subset\bC^d$ is determining for an analytic function $f:\Omega\to\overline{\bD}$ if whenever an analytic function $g:\Omega\rightarrow\overline{\bD}$ coincides with $f$ on $\cD$, equals to $f$ on whole $\Omega$. This note finds several sufficient conditions for a subset of the symmetrized bidisk to be determining. For any $N\geq1$, a set consisting of $N^2-N+1$ many points is constructed which is determining for any rational inner function with a degree constraint. We also investigate when the intersection of the symmetrized bidisk intersected with some special algebraic varieties can be determining for rational inner functions.
\end{abstract}

\maketitle

\section{Introduction}
\subsection{Motivation}
For a domain $\Omega$ in $\bC^d$ ($d\geq 1$), let $\bS(\Omega)$ denote the set of analytic functions $f:\Omega\to\overline{\bD}$, where $\bD$ denotes the open unit disk in $\bC$. Given a function $f\in \bS(\Omega)$, this paper revolves around the question when a given subset $\cD$ of $\Omega$ has the property that whenever $g\in\bS(\Omega)$ coincides with $f$ on $\cD$, equals to $f$ on whole $\Omega$. When a subset has this property we call it a {\em determining set} for $(f,\Omega)$, or just $f$ when the domain is clear from the context. For example, $\{0,1/2\}$ is a determining set for the identity map (by the Schwarz Lemma); any open subset of $\Omega$ is determining for any analytic function on $\Omega$ (by the Identity Theorem). See Rudin \cite[Chapter 5]{Rudin} for some interesting results related to a similar concept for $\Omega=\bD^d$.

The motivation behind the study of determining sets comes from the Pick interpolation problem. It corresponds to the case when $\cD$ is a finite set. Given a finite subset $\cD=\{\lambda_1,\lambda_2,\dots,\lambda_N\}$ of $\Omega$ and points $w_1,w_2,\dots,w_N$ in the open unit disk $\bD$, the Pick interpolation problem asks if there is an analytic function $f:\Omega\to\bD$ such that $f(\lambda_j)=w_j$ for $j=1,2,\dots,N$. Therefore in this case, $\cD$ being a determining set for $(f,\Omega)$ means that the (solvable) Pick problem $\lambda_j\mapsto f(\lambda_j)$ has a unique solution. In view of Pick's pioneering work \cite{PICK}, it is therefore clear that when $\Omega=\bD$, then $\cD$ is determining for $f$ if and only if the Pick matrix 
$$
\begin{bmatrix}
\frac{1-f(\lambda_i)\overline{f(\lambda_j)}}{1-\lambda_i\overline{\lambda_j}}
\end{bmatrix}_{i,j=1}^N
$$has rank less than $N$, which is further equivalent to the existence of a Blaschke function of degree less than $N$ solving the data. The classical Pick interpolation problem has seen a wide range of generalizations. To mention a few, it is known when a given Pick data is solvable when $\Omega$ is the polydisk $\bD^d$ \cite{AM-Poly}, the Euclidian ball $\bB_d$ \cite{KZ-Tran}, the symmetrized bidisk \cite{SYM_Real, Tirtha-Hari-JFA}, an affine variety \cite{JK-JFA} and in more general setting of test functions \cite{DM, DMM}. However, unlike the classical case, it is rather obscure in higher dimension when it comes to understanding when a given solvable Pick problem has a unique solution, and usually one has to settle with either necessary or sufficient conditions; see for example \cite{AM_NYJM, DS-CAOT, DS-1, DS-2}.

\subsection{The main results}
The purpose of this article is to explore this direction where the underlined domain is the {\em symmetrized bidisk}
\begin{align}\label{symmetrized bidisk}
\mathbb{G}:=\{(z_1+z_2, z_1 z_2): (z_1, z_2)\in\mathbb{D}^2\}.
\end{align}Following the work \cite{AYSym} of Agler and Young, this domain has remained a field of extensive research in operator theory and complex geometry constituting examples and counter-examples to celebrated problems in these areas such as the rational dilation problem \cite{AY-2003,BSS} and the Lempert Theorem \cite{Costara}. In quest of understanding the determining sets we shall actually consider the following more general situation.

\begin{definition}\label{Strong-Pick-Set}
Let $\Omega\subset\bC^d$ be a domain, $E\subset\Omega$ and $f\in\bS(\Omega)$. We say that a subset $\cD$ of $E$ is {\em determining} for $(f, E)$ if for every $g\in\bS(\Omega)$, $g=f$ on $\cD$ implies $g=f$ on $E$. If $\cD$ is determining for $(f,E)$ for all $f\in\bS(\Omega)$, then we say that $\cD$ is determining for $E$. Moreover, when $E$ is the largest set in $\Omega$ such that $\cD$ is determining for $(f,E)$, we say that $E$ is the {\em uniqueness set} for $(f,\cD)$, i.e., in this case
$$
E=\bigcap\{Z(g-f): g\in\bS(\Omega)\mbox{ and } g=f \mbox{ on }\cD\}.
$$Here, for a function $f$, we use the standard notation $Z(f)$ for the zero set of $f$.
\end{definition}
Note that if $E$ is the uniqueness set for $(f,\cD)$, then for every $z\in \Omega\setminus E$, there exists a function $g\in \bS(\Omega)$ such that $g=f$ on $\cD$ but $f(z)\neq g(z)$. Remarkably, when $\cD$ is a finite subset of $\bG$, then for any function $f\in\bS(\bG)$, the uniqueness set for $(f,\cD)$ is an affine variety (see \cite{DKS}, \cite{AM}). This is owing to the fact that every solvable Pick data in $\bG$ always has a rational inner solution (see \cite{DKS}, \cite{AM-Bidisc}). Also note that if $f$ and $g$ agree on $\cD$, then $\cD$ is determining for $(f,E)$ if and only if $\cD$ is determining for $(g,E)$ also. In view of these facts, we shall mostly be concerned with the case when the function $f$ in Definition \ref{Strong-Pick-Set} is rational and inner. Here, a function $f$ in $\bS(\bG)$ is called {\em inner}, if $\lim_{r\to 1-}|f(r\zeta_1+r\zeta_2,r^2\zeta_1\zeta_2)|=1$ for almost all $\zeta_1,\zeta_2$ in $\bT$.

Note that $\bG$ is the image of $\bD^2$ under the (proper) holomorphic map $\pi:(z_1,z_2)\mapsto (z_1+z_2,z_1z_2)$. The topological boundary of $\bG$ is $\partial\bG:=\pi(\overline{\bD}\times \bT)\cup\pi(\bT\ \times\bD)$ and the distinguished boundary of $\bG$ is $b\bG:=\pi(\bT\times \bT)$ (see \cite{SYM_GEO}). Here the {\em distinguished boundary} of a bounded domain $\Omega \subset\bC^d$ is the $\check{\mbox S}$ilov boundary with respect to the algebra of complex-valued functions continuous on $\overline{\Omega}$ and holomorphic in $\Omega$. A special type of algebraic varieties has been prevalent in the study of uniqueness of the solutions of a Pick interpolation problem (see \cite{AM, DKS, K, KZ_JGA, KZ-Tran, KM}). We define it below.
\begin{definition}\label{D:DistVar}
An algebraic variety $Z(\xi)$ in $\bC^2$ is said to be {\em distinguished} with respect to a bounded domain $\Omega$, if
$$Z(\xi)\cap\Omega\neq \emptyset \quad\text{ and }\quad Z(\xi)\cap\partial \Omega=Z(\xi)\cap b\Omega.$$
\end{definition}
An example of a distinguished variety with respect to $\bG$ is $\{(2z,z^2):z\in\bC\}$. We refer the readers to the papers \cite{AM, BKS-APDE, DKS, DAS-SARKAR, Pal-Shalit} for results concerning these varieties and their connection to interpolation problems. 


We now state the main results of this paper in the order they are proved.
\begin{enumerate}
	\item In \S \ref{SS:GenDom} we reformulate the notion of determining set in the more general setting of reproducing kernel Hilbert spaces and find a sufficient condition for a finite subset of a general domain to be determining. This is Theorem \ref{T:Unique-Sl}. We also show by an example that the sufficient condition need not be necessary, in general.
	
	\item Starting with a natural number $N$, \S \ref{SS:FiniteSet} constructs a finite subset of $\bG$ consisting exactly of $N^2-N+1$ many points which is determining for any rational inner function with a natural degree constraint on it. This is Theorem \ref{Main-Result2}. Proposition \ref{LE-1} is an intermediate step of the construction and is interesting by its own.
	
\item Given a distinguished variety $\cW=Z(\xi)$, we investigate in \S \ref{SS:DistVar} when  the intersection $\cW\cap\mathbb{G}$ can be the uniqueness set for $(f,\cD)$, where $f$ is a rational inner function and $\cD$ a finite subset of $\bG$ -- see Theorem \ref{MainThm}. The preparatory results Propositions \ref{Main_2} and \ref{Main3} are interesting in their own rights. Propositions \ref{Main_2} states that if $f$ is a rational inner function with some regularity assumption, then there is a natural number $N$ depending on $f$ large enough so that {\em any} subset of $\cW\cap\bG$ consisting of $N$ points is determining for $(f,\cW\cap\bG)$. This section then goes on to find (in Theorem \ref{Main4}) a sufficient condition for $\cW\cap\bG$ to be determining for a rational inner function $f$ with a regularity assumption on it. The condition is just that the inequality
\begin{align*}
\quad 2\operatorname{Re}\langle f, \xi h\rangle_{H^2}< \|\xi h\|^2_2
\end{align*} holds, whenever $h$ is a non-zero analytic function on $\bG$ and $\xi h$ is bounded on $\bG$. Here the inner product is the Hardy space inner product, briefly discussed in \S \ref{SS:DistVar}.

\item \S \ref{S:BddExt} proves a bounded extension theorem for distinguished varieties with no singularities on $b\bG$. More precisely, given a distinguished variety $\cW$, we show that corresponding to every two-variable polynomial $f$, there is a rational function $F$ on $\bG$ such that $F|_{\cW\cap\bG}=f$ and that $\sup_\bG|F(s,p)|\leq \alpha \sup_{\cW\cap\bG}|f|$, for some constant $\alpha$ depending only on the distinguished variety $\cW$.
\end{enumerate} 

\section{Determining and the uniqueness sets}
\subsection{A result for a general domain}\label{SS:GenDom}
We begin by proving a sufficient condition for a finite subset of a general domain to be determining. The concept of determining set can be formulated in a general setup of reproducing kernel Hilbert spaces. Here a {\em kernel} on a domain $\Omega$ in $\bC^d$ ($d\geq 1$) is a function $k:\Omega\times\Omega\to\bC$ such that for every choice of points $\lambda_1,\lambda_2,\dots,\lambda_N$ in $\Omega$, the $N\times N$ matrix $[k(\lambda_i,\lambda_j)]$ is positive definite. Given a kernel $k$, there is a unique Hilbert space $H(k)$ associated to it, called the reproducing kernel Hilbert space; we refer the uninitiated reader to the book \cite{Paulsen}. For the purpose of this paper, all that is needed to know is that elements of the form $\{\sum_{j=1}^nc_jk(\cdot,\lambda_j):c_j\in\bC\mbox{ and }\lambda_j\in\Omega\}$ constitute a dense set of $H(k)$. A kernel $k$ is said to be a holomorphic kernel, if it is holomorphic in the first and conjugate holomorphic in the second variable. Note that when $k$ is holomorphic, then so are the elements of $H(k)$. Let us denote by $\operatorname{Mult}H(k)$ the algebra of all bounded holomorphic functions $\varphi$ on $\Omega$ such that $\varphi\cdot f \in H(k)$ whenever $f\in H(k)$. Such a holomorphic function is generally referred to as a {\em multiplier} for $H(k)$. Let $\operatorname{Mult}_1H(k)$ denote the set of all multipliers $\varphi$ such that the operator norm of $M_\varphi: f\mapsto \varphi\cdot f$ for all $f$ in $H(k)$ is no greater than one. A subset $\cD\subset\Omega$ is said to be {\em determining} for a function $\varphi$ in $\operatorname{Mult}_1H(k)$ if whenever $\psi\in\operatorname{Mult}_1H(k)$ such that $\varphi=\psi$ on $\cD$, then $\varphi=\psi$ on $\Omega$.
\begin{thm}\label{T:Unique-Sl}
Let $k$ be a holomorphic kernel on a domain $\Omega$ in $\bC^d$, $\varphi\in \operatorname{Mult}_1H(k)$ and $\cD=\{\lambda_1,\lambda_2,\dots,\lambda_N\}\subset\Omega$. If the matrix
\begin{align}\label{PickMat}
[(1-\varphi(\lambda_i)\overline{\varphi(\lambda_j)})k(\lambda_i,\lambda_j)]_{i,j=1}^N
\end{align}
is singular, then $\cD$ is determining for $\varphi$.
\end{thm}
\begin{proof}
Since the matrix \eqref{PickMat} is singular, there is a non-zero vector in its kernel; let us denote it by $\gamma$. Let  $\lambda_{N+1}$ be any point in $\Omega\setminus\cD$, and $\psi\in\operatorname{Mult}_1H(k)$ be any function such that $\varphi=\psi$ on $\cD$. Since $\psi\in \operatorname{Mult}_1H(k)$, the operator $M_\psi:f\mapsto \psi\cdot f$ is a contractive operator on $H(k)$ and therefore for every $z\in\bC$,
$$\langle[(1-\psi(\lambda_i)\overline{\psi(\lambda_j)})k(\lambda_i,\lambda_j)]_{i,j=1}^{N+1}\begin{bmatrix}
\gamma \\
z
\end{bmatrix}, \begin{bmatrix}
\gamma \\
z
\end{bmatrix}\rangle\geq{0}.$$
Since $\gamma\in\operatorname{Ker}[(1-\varphi(\lambda_i)\overline{\varphi(\lambda_j)})k(\lambda_i,\lambda_j)]$ and $\varphi=\psi$ on $\cD$, the above inequality collapses to 
$$2\operatorname{Re}[\overline{z}\sum_{j=1}^N(1-\overline{\psi(\lambda_j)}\psi(\lambda_{N+1}))\gamma_jk(\lambda_{N+1},\lambda_j)]+ |z|^2(1-|\psi(\lambda_{N+1})|^2)||k_{\lambda_{N+1}}||^2\geq{0}.$$
Since the above inequality is true for all $z\in\mathbb{C}$, we have
$$\sum_{j=1}^N(1-\overline{\psi(\lambda_j)}\psi(\lambda_{N+1}))\gamma_jk_{N+1,j}=0.$$
Therefore, we have an implicit formula for $\psi(\lambda_{N+1})$:
\begin{align}\label{formula}
\psi(\lambda_{N+1})\left(\sum_{j=1}^N\overline{\psi(\lambda_j)}\gamma_jk(\lambda_{N+1}, \lambda_j\right)=\sum_{j=1}^N\gamma_jk(\lambda_{N+1},\lambda_j).
\end{align}
Define for $z$ in $\Omega$,
$$L(z)=\sum_{j=1}^N\gamma_jk_{\lambda_j}(z)=\sum_{j=1}^N\gamma_jk(z,\lambda_j).$$ By definition, it is clear that $L\in H(k)$. Consider the open set $\cO=\Omega\setminus Z(L)$. Note that if $\lambda_{N+1}\in\cO$, then the right hand side of equation \eqref{formula} does not vanish, and therefore $\psi(\lambda_{N+1})$ is uniquely determined. 

Now suppose $\phi=\psi$ on $\cO$. By the assumption that $\cO$ is a set of uniqueness for $\operatorname{Mult}_1(H(k))$, it follows that $\phi=\psi$.
\end{proof}
The converse of the above result is not true as the simple example below demonstrates.
\begin{example} Let $k$ be the Bergman kernel on $\Omega=\mathbb{D}$, i.e., $k(z, w)= (1-z\overline{w})^{-2}$. Then it is well-known that $\operatorname{Mult}_1H(k)=\bS(\bD)$, see for example \cite[section 2.3]{AM-Book}. By the Schwarz Lemma, $\cD=\{0,1/2\}$ is determining for the identity function. However, the matrix
$\sbm{
1 & 1 \\
1 & 4/3
}$
is non-singular.  
\end{example}

The rest of the paper specializes to the symmetrized bidisk.
\subsection{Finite sets as a determining set}\label{SS:FiniteSet}
Given a natural number $N$, this subsection constructs a finite subset $\cD$ of $\bG$ consisting exactly of $N^2-N+1$ many points, which is determining for any rational inner function on $\bG$ with a degree constraint on it. This is inspired by the work of Scheinker \cite{DS-CAOT} which extends the following classical result for the unit disk to the polydisks.
\begin{lemma}[Pick \cite{PICK}]\label{Pick}
Let $\cD=\{\lambda_1, \lambda_2,\dots,\lambda_N\}\subset\bD$ and $f$ be a rational inner function on $\mathbb{D}$ with degree strictly less than $N$. Then if $g\in\bS(\mathbb{D})$ is such that $f=g$ on $\cD$, then $f=g$ on $\bD$.
\end{lemma}

For $\epsilon>0$ and $z\in\bC$, let $D(z;\epsilon):=\{w\in{\mathbb{D}}: |z-w|<{\epsilon}\}$. For $\zeta\in\bT$ and $a\in\bD$, let $m_{\zeta,a}$ be the M\"obius map 
$$m_{\zeta,a}(z)=\zeta\frac{z-a}{1-\overline{a}z}.$$
We shall have use of two notions of degree for a polynomial in two variables. The one used in this subsection is the following. For a polynomial $\xi(z,w)=\sum_{i,j}a_{i,j}z^iw^j$, we define $\deg\xi:=\max (i+j)$ such that $a_{i,j}\neq0$. The degree of a rational function in its reduced fractional representation is defined to be the degree of the numerator polynomial. The following is an intermediate step to proving Theorem \ref{Main-Result2}.
\begin{proposition}\label{LE-1}
Let $N$ be a positive integer and for each $j=1,2,\dots,N$, $\beta_j$ be distinct points in ${\mathbb{T}}$ and $D_j$ be the analytic disks $D_j=\{(z+\beta_jz, \beta_j z^2):z\in\bD\}$. Then
\begin{enumerate}
\item [(a)] there exist $\beta\in\bT$ and $\epsilon>0$ such that for every fixed $\zeta\in D(\beta;\epsilon)\cap\mathbb{T}$ and $a\in{D(0;\epsilon)}$, the analytic disk
$$
\cD_{\zeta,a}=\{(z+m_{\zeta,a}(z), zm_{\zeta, a}(z)):z\in\bD\}
$$intersects each of the analytic disks $D_j$ at a non-zero point; 
\item[(b)] for each $\zeta\in\bT$ and $\epsilon>0$, the set
$$
\cD_\zeta=\{(z+m_{\zeta,a}(z), zm_{\zeta, a}(z)):z\in\bD \text{ and } a\in{D(0;\epsilon)}\}
$$
is a determining set for any analytic functions on $\mathbb{G}$; and
\item[(c)] the set 
$$E=\{(z+\beta_jz, \beta_j z^2): z\in\bD \text{ and } j=1,2,\dots,N\}=\cup_{j=1}^ND_j$$is a determining set for any rational inner function of degree less than $N$.
\end{enumerate}
\end{proposition}
\begin{proof}
For part (a), note that given a $\zeta\in\bT$ and $a\in\bD$, the analytic disk $\cD_{\zeta,a}$ intersects each $D_j$ at a non-zero point if and only if there exist $0\neq z\in\bD$ such that for each $j$, $\beta_jz=m_{\zeta,a}(z)$, which is equivalent to having $\overline{a}\beta_jz^2+(\beta_j-\zeta)z-a\zeta=0$. Therefore $\zeta$ must belong to $\bT\setminus\{\beta_j: j=1,2,\dots,N\}$. Now fix one such $\zeta$ and $j$. Let $\lambda_1(a),\lambda_2(a)$ be the roots of the polynomial above. Then clearly $\lambda_1(0)=0=\lambda_2(0)$. Therefore by continuity of the roots, there exists $\epsilon>0$ such that whenever $a\in D(0;\epsilon)$, $\lambda_1(a)$ and $\lambda_2(a)$ belong to $\bD$. This $\epsilon$ will of course depend on $j$ but since there are only finitely many $j$, we can find an $\epsilon>0$ so that (a) holds.

For part (b) we have to show that if $f:\mathbb{G}\to\overline\bD$ is any analytic function such that $f|_{\cD_\zeta}=0$, then $f=0$ on $\bG$. Fix $z\in\mathbb{D}$ and consider $f_z:\mathbb D\to\overline\bD$ defined by $f_z:w\mapsto f(z+w,zw)$. Since $f$ vanishes on $\mathcal D_\zeta$, $f_z$ vanishes on $\{m_{\zeta,a}(z):a\in D(0;\epsilon)\}$ which shows that $f_z=0$ on $\bD$. Since $z\in\bD$ is arbitrary, $f=0$ on $\bG$.

For (c), let $f$ be a rational inner function of degree less than $N$ and $g\in{\bS(\mathbb{G})}$ be such that $g=f$ on each $D_j$. For each $\zeta$ and $a$ as in part (a), $\cD_{\zeta,a}$ intersects each $D_j$ at say $(s_j,p_j)=(\lambda_j+m_{\zeta,a}(\lambda_j),\lambda_jm_{\zeta,a}(\lambda_j))$. Restrict $f$ and $g$ to $\cD_{\zeta,a}$ to get $f_{\zeta,a}(z)=f(z+m_{\zeta,a}(z), zm_{\zeta, a}(z))$ and $g_{\zeta,a}(z)=g(z+m_{\zeta,a}(z), zm_{\zeta, a}(z))$. Then clearly $f_{\zeta,a}$ is a rational inner function on $\bD$ of degree less than $N$ and $g_{\zeta,a}\in\bS(\bD)$. Then for each $j=1,2,\dots, N$, $g_{\zeta,a}(\lambda_j)=f_{\zeta,a}(\lambda_j)$. Therefore by Lemma \eqref{Pick}, we have $g_{\zeta,a}=f_{\zeta,a}$ on $\mathbb{D}$ for each $\zeta$ and $a$ as in part (a). Hence $g=f$ on $\cD$, which by part (b) gives $g=f$ on $\bG$. This completes the proof.
\end{proof}

\begin{thm}\label{Main-Result2}
For any $N\geq1$, there exists a set $D$ consisting of $(N^2-N+1)$ points in $\mathbb{G}$ such that $\cD$ is a determining set for any rational inner function of degree less than $N$.
\end{thm}
\begin{proof}For $N=1$, it is trivial because then a rational inner function of degree less than $1$ is identically constant. So suppose $N>1$. Let $\lambda_1:=0, \lambda_2,...,\lambda_N$ be distinct points in $\mathbb{D}$, $\beta_1,..., \beta_{N}$ be distinct points in $\mathbb{T}$ and $D_1,...,D_{N}$ be the analytic disks as in Proposition \ref{LE-1}. Consider the set
$$\cD=\{(\lambda_j+\beta_k\lambda_j, \beta_k\lambda_j^2): k,j=1,2,...,N\}.$$ 
Since $\beta_j$ and $\lambda_j$ are distinct, $\cD$ consists of precisely $N^2-N+1$ many points. Let $f$ be a rational inner function on $\bG$ and $g\in\bS(\bG)$ be such that $g$ agrees with $f$ on $\cD$. As before restrict $f$ and $g$ to each $D_k$ to obtain rational inner functions $f_{k}(z)=f(z+\beta_kz,z^2\beta_k)$ and $g_k(z)=g(z+\beta_kz,z^2\beta_k)$ on the unit disk $\bD$. We then have
$f_k(\lambda_j)=g_k(\lambda_j)$ for each $j=1,2,\dots,N$. Thus by Lemma \ref{Pick}, $f_k(z)=g_k(z)$ on $\bD$ for each $k=1,2,\dots,N$, which is same as saying that $f=g$ on $\cup_{k=1}^ND_k$. Consequently, by part (c) of Proposition \ref{LE-1}, $f=g$ on $\bG$.
\end{proof}

\subsection{Distinguished varieties as a determining and the uniqueness set}\label{SS:DistVar}
A rational function $f=g/h$ with relatively prime polynomials $g$ and $h$, is called {\em regular} if $h\neq{0}$ on $\overline{\mathbb{G}}$. For example, note that while the rational function $(3p-s)/(3-s)$ is regular, $(2p-s)/(2-s)$ is not.

We first recall the known results that will be used later. Let $\cW=Z(\xi)$ be a distinguished variety with respect to $\mathbb{G}$. Then it follows easily that $\cV=Z(\xi\circ\pi)$ defines a distinguished variety with respect to $\bD^2$. Lemma 1.2 of \cite{AM} produces a regular Borel measure $\nu$ on $\partial\cV:=\cV\cap\bT^2$ such that $\nu$ gives rise to a Hardy-type Hilbert function space on $\cV\cap\bD^2$, denoted by $H^2(\nu)$, i.e., $H^2(\nu)$ is the closure in $L^2(\nu)$ of polynomials such that evaluation at every point in $\cV\cap\bD^2$ is a bounded linear functional on $H^2(\nu)$. It was then shown in \cite[Lemma 3.2]{Pal-Shalit} that the push-forward measure $\mu(E)=\nu(\pi^{-1}(E))$ for every Borel subset $E$ of $\partial\cW:=\cW\cap b\Gamma$ has all the properties that $\nu$ has. Furthermore, the spaces $H^2(\mu)$ and $H^2(\nu)$ are unitary equivalent via the isomorphism given by
\begin{align}\label{Unitary}
 U:H^2(\mu)\rightarrow{H^2{(\nu)}}\quad \text{ by } \quad U:f\mapsto f\circ\pi.
\end{align}Note that if $k^\mu$ and $k^\nu$ are the Szeg\"o-type reproducing kernels for $H^2(\mu)$ and $H^2(\nu)$, respectively, then for every $(z,w)\in\cV\cap\bD^2$ and $f\in H^2(\mu)$
$$\langle U^*k^\nu_{(z,w)}, f\rangle_{H^2(\mu)}=\langle k^\nu_{(z,w)}, Uf\rangle_{H^2(\nu)}=f\circ\pi(z,w)=\langle k^\mu_{\pi(z,w)}, f\rangle_{H^2(\mu)}.$$
We observe the following.
\begin{lemma}\label{Measure}
Let $\cW$ be a distinguished variety with respect to $\mathbb G$ and $\mu$ be the regular Borel measure on $\partial\cW$ as in the preceding discussion. Then for every regular rational inner function $f$ on $\mathbb{G}$, the multiplication operator $M_f$ on $H^2(\mu)$ has a finite dimensional kernel.

\end{lemma}
\begin{proof}
We note that for every $(z,w)\in\cV\cap\bD^2$,
\begin{align*}
U^*M_{f\circ\pi}^*k^\nu_{(z,w)}=\overline{f\circ\pi(z,w)}U^*k^\nu_{(z,w)}=\overline{f\circ\pi(z,w)}k^\mu_{\pi(z,w)}=M_f^*k^\mu_{\pi(z,w)}=M_{f}^*U^*k^\nu_{(z,w)}.
\end{align*}
Thus $M_f$ on $H^2(\mu)$ and $M_{f\circ\pi}$ on $H^2(\nu)$ are unitarily equivalent via the unitary $U$ as in \ref{Unitary}. Now the lemma follows from \cite[Theorem 3.6]{DS-1}, which states that $\operatorname{Ker}M_{f\circ \pi}$ is finite dimensional.
\end{proof}

\begin{proposition}\label{Main_2}
Let $\mathcal W=Z(\xi)$ be a distinguished variety with respect to $\mathbb G$ and $f$ be a regular rational inner function on $\mathbb G$. If $\dim\operatorname{ Ker} M_{f}^* < N$, then any $N$ distinct points in $\cW\cap\bG$ is a determining set for $(f, \cW\cap\mathbb{G})$. 
\end{proposition}

\begin{proof}
Let $\{w_1,w_2,\dots,w_N\}$ be distinct points in $\cW\cap\bG$ and $g\in\bS(\bG)$ be such that $g(w_j)=f(w_j)$ for each $j=1,2,\dots,N$. Let $\cV=Z(\xi\circ\pi)$ and $\{v_1,v_2,\dots,v_N\}$ be in $\cV\cap\bD^2$ such that $\pi(v_j)=w_j$ for all $j=1,2,...,N$. Thus $g\circ\pi (v_j)=f\circ\pi(v_j)$ for each $j=1,2,\dots,N$. Theorem 1.7 of \cite{DS-1} yields $g\circ\pi=f\circ\pi$ on $\cV\cap\bD^2$ which is same as $g=f$ on $\cW\cap\bG$. This completes the proof.
\end{proof}

The $2$-degree of a two-variable polynomial $\xi\in\mathbb{C}[z,w]$ is defined as $(d_1, d_2)=:\operatorname{2-deg}\xi$, where $d_1$ and $d_2$ are the largest power of $z$ and $w$, respectively in the expansion of $\xi(z,w)$. The reflection of a two-variable polynomial $\xi\in\mathbb{C}[z,w]$ is defined as 
$$
\widetilde{\xi}(z,w)=z^{d_1}w^{d_2}\overline{\xi\big(\frac{1}{\overline{z}},\frac{1}{\overline{w}}\big)}.
$$For a rational function $f(z,w)=\xi(z,w)/\eta(z,w)$ with $\xi$ and $\eta$ having no common factor, the $2$-degree of $f$ is defined to be the $2$-degree of the numerator. For two pairs of non-negative integers $(p,q)$ and $(m,n)$, we write $(p,q)\leq(m,n)$ to indicate that $p\leq m$ and $q\leq n$.

\begin{proposition}\label{Main3}
Let $\cW=Z(\xi)$ be an irreducible distinguished variety and $f$ be a regular rational inner function on $\mathbb{G}$ of the form
\begin{align}\label{RatInnG}
f\circ\pi(z,w)=(zw)^m\frac{\widetilde{\eta\circ\pi}(z,w)}{\eta\circ\pi(z,w)}.
\end{align} If $\operatorname{2-deg}\xi\circ\pi \leq \operatorname{2-deg}f\circ\pi$, then for each $(s,p)\in\mathbb{G}\setminus(\mathbb{G}\cap\cW)$ there exists a regular rational inner function $g$ on $\bG$ such that $g$ coincides with $f$ on $\cW\cap\bG$ but $g(s,p)\neq f(s,p)$.
\end{proposition}
\begin{proof}
Let $\operatorname{2-deg}\eta\circ\pi=(l,l)$ and $\operatorname{2-deg}\xi\circ\pi=(n,n)$. The hypothesis then is that $m+l-n$ is non-negative. For $\epsilon>0$, define a symmetric function $g_\epsilon$ on $\bD^2$ as
\begin{align}\label{Rat}
g_{\epsilon}(z,w)=
\frac{(zw)^m\widetilde{\eta\circ\pi}(z,w)+\epsilon \widetilde{\xi\circ\pi}(z,w)}{\eta\circ\pi(z,w)+\epsilon(zw)^{m+l-n}\xi\circ\pi(z,w)}.
\end{align}
Simple computation shows that the reflection of the denominator of $g_\epsilon$ is equal to the numerator of $g_\epsilon$, which implies that each each $g_\epsilon$ is a rational inner function on $\bD^2$ provided that the denominator does not vanish on $\bD^2$. Since $\eta\circ\pi$ does not vanish on $\overline{\bD}^2$, we can always find a sufficiently small $\epsilon$ so that the denominator of each $g_\epsilon$ does not vanish in $\overline{\bD}^2$, thus making $g_\epsilon$ regular.

By Proposition 4.3 of \cite{Knese-Trans}, $\xi\circ\pi=c \widetilde{\xi\circ\pi}$ for some $c\in\bT$. This ensures that each $g_\epsilon$ coincides with $f$ on $\cW\cap\bG$. Now let $(z_0,w_0)\in\bD^2$ be such that $\pi(z_0,w_0)\in\bG\setminus\cW$. Then $g_\epsilon(z_0,w_0)=f\circ\pi(z_0,w_0)$
if and only if 
 $$\frac{(z_0w_0)^m\widetilde{\eta\circ\pi}(z_0,w_0)+\epsilon \Bar{c}\xi\circ\pi(z_0,w_0)}{\eta\circ\pi(z_0,w_0)+\epsilon(z_0w_0)^{m+l-n}\xi\circ\pi(z_0,w_0)}=(z_0w_0)^m\frac{\widetilde{\eta\circ\pi}(z_0,w_0)}{\eta\circ\pi(z_0,w_0)}, $$ which, after cross-multiplication and using the fact that $\xi\circ\pi(z_0,w_0)\neq0$, leads to
\begin{align}\label{eq_g}
\overline{c}\eta\circ\pi(z_0,w_0)=(z_0w_0)^{2m+l-n}\widetilde{\eta\circ\pi}(z_0,w_0).
  \end{align}
Since $\eta\circ\pi$ does not vanish on $\overline{\bD}^2$, we have $z_0w_0\neq0$. Therefore the above equation holds if and only if
\begin{align}\label{AuxEqn}
f\circ\pi(z_0,w_0)=(z_0w_0)^m\frac{\widetilde{\eta\circ\pi}(z_0,w_0)}{\eta\circ\pi(z_0,w_0)}=\frac{\overline{c}}{(z_0w_0)^{m+l-n}}.
\end{align} 
If $m+l-n=0$, then $f$ is a constant function. The hypothesis on the 2-degrees of $\xi$ and $f$ then implies that $\xi$ must be constant. This is not possible because $\xi$ defines a distinguished variety. Therefore $m+l-n\geq 1$, in which case, equation \eqref{AuxEqn} implies that $|f\circ\pi(z_0,w_0)|>1$. This again is a contradiction because $f$ is a rational inner function and so by the Maximum Modulus Principle, $|f\circ\pi(z)|\leq 1$ for every $(z,w)\in\bD^2$. Consequently, $g_\epsilon(s,p)\neq f(s,p)$ for every $(s,p)\in \bG\setminus (\cW\cap\bG)$.
\end{proof}
\begin{remark}
In a forthcoming paper \cite{BK} it is shown that any rational inner function on $\bG$ is of the form \eqref{RatInnG} possibly multiplied by a unimodular constant.
\end{remark}
\begin{thm}\label{MainThm}
Let $\cW=Z(\xi)$ be an irreducible distinguished variety with respect to $\mathbb{G}$, $f$ be a regular rational inner function on $\bG$ of the form \eqref{RatInnG} such that $\operatorname{2-deg}\xi\circ\pi \leq \operatorname{2-deg}f\circ\pi$, and $\cD$ be any subset of $\cW\cap\mathbb{G}$ consisting of at least $1+\dim\operatorname{Ker}M_f^*$ many points. Then $\cW\cap\mathbb{G}$ is the uniqueness set for $(f,\cD)$.
 \end{thm}
 \begin{proof}
Consider the multiplication operator $M_{f}$ on $H^2(\mu)$, where $H^2(\mu)$ is the Hilbert space corresponding to $\cW$ as mentioned in Lemma \ref{Measure}. By this lemma, $\dim\operatorname{Ker}(M_{f}^*)$ is finite. So let $N$ be such that $\dim\operatorname{Ker}(M_{f}^*)<N$ and $\cD=\{\lambda_1, \lambda_2,..., \lambda_N\}\subset\cW\cap\mathbb{G}$. By Proposition \ref{Main_2}, $\cD$ is determining for $(f,\cW\cap\mathbb{G})$. We use Proposition \ref{Main3} to show that $\cW\cap\mathbb{G}$ is the uniqueness set. Toward that end, pick $(s,p)\in\bG\setminus\cW\cap\bG$. Proposition \ref{Main3} guarantees the existence of a (regular) rational inner function $g$ that coincides with $f$ on $\cW\cap\bG$ but $g(s,p)\neq f(s,p)$. This proves that $\cW\cap\mathbb{G}$ is the uniqueness set for the interpolation problem. This completes the proof of the theorem.
\end{proof}

\begin{remark} An {\em extremal} interpolation problem in $\bG$ is a solvable problem with no solution of supremum norm less than $1$. Let $\cD=\{\lambda_1,\lambda_2,\dots,\lambda_N\}$ be a subset of $\bG$ and $f$ be a rational inner function on $\bG$ such that the $N$-point Pick problem $\lambda_j\mapsto f(\lambda_j)$ is extremal and that none of the $(N-1)$-point subproblems is extremel. Then it is shown in \cite{DKS} that the uniqueness set for $(f,\cD)$ contains a distinguished variety. Theorem \ref{MainThm} can be seen as a converse to this result. Indeed, Theorem \ref{MainThm} starts with a distinguished variety $\cW=Z(\xi)$ and produces a regular rational inner function $f$ and a finite set $\cD$ depending on $\cW$ such that $\cW\cap\bG$ is the uniqueness set for $(f,\cD)$. In addition, we note that the problem $\lambda_j\mapsto f(\lambda_j)$ is an extremal problem. This is because if $g$ is any solution of the problem, then by Proposition \ref{Main_2} $g=f$ on $\cW\cap\bG$. Thus
$$\|g\|_{\infty,\bG}\geq \|g\|_{\infty,\cW\cap\bG}=\|f\|_{\infty,\cW\cap\bG}=1.$$
The last equality follows because $f$ is a regular rational inner function.
\end{remark}

There is a sufficient condition for a distinguished variety to be determining. In the theorem below and in its proof, the inner product $\langle ,\rangle_{H^2}$ for analytic functions $f,g:\bG\to\bC$ is defined to be
\begin{align}\label{InnProd}
\langle f,g\rangle_{H^2}= \sup_{0<r<1}\int_{\bT\times\bT} f\circ\pi(r\zeta_1,r\zeta_2)\overline{g\circ\pi(r\zeta_1,r\zeta_2)}|J(r\zeta_1,r\zeta_2)|^2dm(\zeta_1,\zeta_2),
\end{align}
where $m$ is the standard normalized Lebesgue measure on $\bT\times\bT$, and $J(z,w)=z-w$ is the Jacobian of the map $\pi:(z,w)\mapsto (z+w,zw)$. See the papers \cite{BDSIMRN,  Tirtha-Hari-JFA,MSRZ} for some motivation for and operator theory on the spaces of analytic functions for which $\|f\|_2:=\sqrt{\langle f,f\rangle}_{H^2}<\infty$. Note here that if $f$ is an inner function on $\bG$, then $\|f\|_2=1.$
\begin{thm}\label{Main4}
Let $\mathcal{W}=Z(\xi)$ be a distinguished variety such that $\xi=\xi_1.\xi_2\dots\xi_l$ where $\xi_i$ are irrudicible polynomials with $\xi_i$ and $\xi_j$ are co-prime for each $i\neq j$ and $f$ be a regular rational inner function on $\mathbb{G}$. If for each analytic function $h(\not\equiv 0)$ on $\mathbb{G}$,
\begin{align*}
\quad 2\operatorname{Re}\langle f, \xi h\rangle_{H^2}< \|\xi h\|^2_2
\end{align*} holds, whenever $\xi h$ is bounded on $\bG$, then $\cW\cap\bG$ is a determining set for $f$.
\end{thm}
\begin{proof}
We shall use contrapositive argument. So suppose that there exists $g\in{\bS(\mathbb{G})}$ such that $g$ coincides with $f$ on $\cW\cap\bG$ but $g\neq f$. Choose an integer $N$ so that $\dim\operatorname{Ker}M_f^*<N$ and pick $N$ distinct points $\lambda_1,....,\lambda_N\in\cW$. Consider the $N$-point (solvable) Nevanlinna-Pick problem $\lambda_j\mapsto f(\lambda_j)$. By Proposition \ref{Main_2} all the solutions to this problem agree on $\cW\cap\bG$. Since $g\neq f$, there exists a $\lambda_{N+1}\in\mathbb{G}\setminus\cW$ such that $g(\lambda_{N+1})\neq f(\lambda_{N+1})$. Now consider the $(N+1)$-point Nevanlinna-Pick problem $\lambda_j\mapsto g(\lambda_j)$ on $\mathbb{G}$. By \cite[Theorem 5.3]{DKS}, every solvable Nevanlinna-Pick problem in $\bG$ has a rational inner solution. Let $\psi$ be a rational inner solution to the $(N+1)$-point problem $\lambda_j\mapsto g(\lambda_j)$. Since $\psi$, in particular, solves the problem $\lambda_j\mapsto f(\lambda_j)$ for each $j=1,2,\dots,N$, $\psi=f$ on $\cW\cap\bG$. But since $\psi(\lambda_{N+1})=g(\lambda_{N+1})\neq f(\lambda_{N+1})$, $\psi$ is distinct from $f$. Since $\psi=f$ on $\cW\cap\bG$, by the Study Lemma there exists a rational function $h$ such that $f-\psi=\xi h$, see \cite[chapter 1]{Alg}. Since $\psi$ is inner,
$$1=\|\psi\|_2^2=\|f-\xi h\|_2^2=\|f\|_2^2-2\operatorname{Re}\langle f, \xi h\rangle_{H^2} +\|\xi h\|_2^2.$$
Since $f$ is an inner function, $\|f\|_{2}=1$, and therefore the above computation leads to $2\operatorname{Re}\langle f, \xi h\rangle= \|\xi h\|^2_2$. This contradicts the hypothesis because $\xi h=f-\psi$ is bounded. Consequently, $g$ must coincide with $f$ on $\bG$. 
\end{proof}
One can easily find examples of distinguished varieties and regular rational inner functions such that the stringent hypothesis of the above result is satisfied.
\begin{example}
Let $f\circ\pi(z,w)=(zw)^d$ and $\cW=Z(\xi)$ be such that 
$$\xi\circ\pi(z,w)=(z^m-w^n)(z^n-w^m),$$ where $m,n$ are mutually prime integers bigger than $d$. Then it follows that $\cW$ is a distinguished variety with respect to $\bG$ because $Z(z^m-w^n)$ is a distinguished variety with respect to $\bD^2$. For concrete example, one can take $d=1$ and $(m,n)=(2,3)$ -- the corresponding distinguished variety then is the Neil parabole. Note that the inner product $\langle,\rangle$ as defined in \eqref{InnProd} can be expressed in terms of the inner product on the Hardy space of the bidisk $H^2(\bD^2)$ as 
\begin{eqnarray}\label{inner-uni}
    \langle f, \xi h\rangle_{H^2(\mathbb{G})}=\frac{1}{\|J\|^2}    \langle J(f\circ\pi), J\big((\xi\circ\pi)( h\circ\pi)\big) \rangle_{H^2(\mathbb{D}^2)}.
\end{eqnarray}
Let $h:\bG\to\bC$ be an analytic function such that $\|\xi h\|_2<\infty$. Since $\{z^iw^j: i,j\geq 0\}$ forms an orthonormal basis for $H^2(\bD^2)$, it is easy to read off from \eqref{inner-uni} that $\langle f,\xi h\rangle=0$. Therefore, by Theorem \ref{Main4}, $\cW\cap\bG$ is a determining set for $f$ as chosen above.
\end{example}

\section{A bounded extension theorem}\label{S:BddExt}
We end with a bounded extension theorem for distinguished varieties with no singularities on the distinguished boundary of $\Gamma$. Here singularity of an algebraic variety $Z(\xi)$ at a point means that both the partial derivatives of $\xi$ vanish at that point. Note that the substance of the following theorem is not that there is a rational extension of every polynomial, it is that the supremum of the rational extension over $\bG$ does not exceed the supremum of the polynomial over the variety intersected with $\bG$ multiplied by a constant that only depends on the variety. See the papers \cite{AAC, Knese-Trans, St} for similar results in other contexts.
\begin{thm}
Let $\cW$ be a distinguished variety with respect to $\mathbb{G}$ such that it has no singularities on $b{\Gamma}$. Then for every polynomial $f\in\mathbb{C}[s,p]$, there exists a rational extension $F$ of $f$ such that
$$|F(s,p)|\leq \alpha\sup_{{\cW\cap\mathbb{G}}}|f|$$
for all $(s,p)\in\mathbb{G}$, where $\alpha$ is a constant depends only on $\cW$.
\end{thm}
\begin{proof}
Let $\cV$ be a distinguished variety with respect to $\mathbb{D}^2$ such that $\cW=\pi(\cV)$. Since $\cW$ has no singularities on $b\Gamma$, it follows that $\cV$ has no singularities on $\mathbb{T}^2$. Invoke Theorem 2.20 of \cite{Knese-Trans} to obtain a rational extension $G$ of the polynomial $f\circ\pi\in\mathbb{C}[z,w]$ such that 
\begin{eqnarray}\label{E-B}
|G(z, w)|\leq \alpha\sup_{\cV\cap\mathbb{D}^2}|f\circ\pi|
\end{eqnarray} 
for all $(z,w)\in\mathbb{D}^2$, where $\alpha$ is a constant depends only on $\cV$.
Now, define a rational function $H$ on $\mathbb{D}^2$ as follows
\begin{eqnarray}\label{Def-H}
H(z,w)= \frac{G(z,w)+G(w, z)}{2}.
\end{eqnarray}
Clearly, $H$ is also a rational extension of $f\circ\pi$ with 
$$|H(z, w)|\leq \alpha\sup_{\cV\cap\mathbb{D}^2}|f\circ\pi| \quad \text {for all } (z,w)\in\mathbb{D}^2.$$ Note that $H$ is a symmetric rational function on $\mathbb{D}^2$. So, there is a rational function $F$ on $\mathbb{G}$ such that $$H(z,w)=(F\circ\pi)(z,w)=F(z+w,zw) \quad \text {for all } (z,w)\in\mathbb{D}^2.$$
Now we will show that this $F$ will do our job. It is easy to see that $F$ is a rational extension of $f$. Let $(s, p)\in\mathbb{G}$. Then there exists a point $(z,w)\in\mathbb{D}^2$ such that $(s, p)=(z+w, zw)$. Now,
\begin{align*}
|F(s,p)|&=|(F\circ\pi)(z, w)|=|H(z,w)|
\leq \alpha\sup_{\cV\cap\mathbb{D}^2}|f\circ\pi|
=\alpha\sup_{\cW\cap\mathbb{G}}|f|.
\end{align*}
This complete the proof.
\end{proof}

\vspace{0.1in} \noindent\textbf{Acknowledgement:}
The first author is supported by the Mathematical Research Impact Centric Support (MATRICS) grant, File No: MTR/2021/000560, by the Science and Engineering Research Board (SERB), Department of Science \& Technology (DST), Government of India.  The second author was supported by the University Grants Commission Centre for Advanced Studies. The research works of the  third author is supported by DST-INSPIRE Faculty Fellowship DST/INSPIRE/04/2018/002458.

The second author thanks his supervisor Professor Tirthankar Bhattacharyya for some fruitful discussions.


\begin{thebibliography}{99}
\bibitem{AAC} K. Adachi, M. Andersson and H. R. Cho, {\em $L^p$ and $H^p$ extensions of holomorphic
functions from subvarieties of analytic polyhedra}, Pacific J. Math. 189 (1999), 201–210. 

\bibitem{AM-Poly} J. Agler, {\em On the representation of certain holomorphic functions define on polydisc}, Topics in operator theory: Ernst D Hellinger memorial volume, 47-66, Oper. Theory Adv. Appl. 48, Birkhauser, Basel, 1990.
		
\bibitem{AM} J. Agler and J. E. McCarthy, {\em Distinguished Varieties}, Acta Math. 194 (2005), no. 2, 133-153.

\bibitem{AM_NYJM} J. Agler and J. E. McCarthy, {\em The three point Pick problem on the bidisk},  New York J. Math. 6 (2000), 227-236.

\bibitem{AM-Bidisc} J. Agler and J. E. McCarthy, {\em Nevanlinna-Pick interpolation on the bidisk}, J. Reine Angew. Math. 506 (1999) 191-204.

\bibitem{AM-Book} J. Agler and J. E. McCarthy, {\em Pick Interpolation and Hilbert Function Spaces}, American Mathematical Society, Providence, 2002.


\bibitem{AYSym}J. Agler and N. J. Young, {\ A commutant lifting theorem for a domain in $\mathbb{C}^2$ and spectral interpolation}, J. Funct. Anal. 161 (1999), no. 2, 452–477.

\bibitem{AY-2003} J. Agler and N. J. Young, {\em A model theory for $\Gamma$-contractions}, J. Operator Theory 49 (2003), no. 1, 45–60. 


\bibitem{SYM_GEO}  J. Agler and N. J. Young, {\em The hyperbolic geometry of the symmetrized bidisc}, J. Geom. Anal. 14 (2004), 375-403.

\bibitem{SYM_Real}  J. Agler and N. J. Young, {\em Realization of functions on the symmetrized bidisc}, J. Math. Anal. Appl. 453 (2017), 227-240.

\bibitem{BKS-APDE} T. Bhattacharyya, P. Kumar and H. Sau, {\em Distinguished varieties through the Berger--Coburn--Lebow theorem}, Anal. PDE 15 (2022), no. 2, 477–506.

\bibitem{BSS} T. Bhattacharyya, S. Pal and S. Shyam Roy,
{\em Dilations of $\Gamma$-contractions by solving operator equations},
Adv. Math. 230 (2012), no. 2, 577–606.

\bibitem{BDSIMRN}  T. Bhattacharyya, B. K. Das and H. Sau, {\em  Toeplitz operators on the symmetrized bidisc}, Int. Math. Res. Not. IMRN 2021, no. 11, 8492–8520.

\bibitem{Tirtha-Hari-JFA} T. Bhattacharyya and H. Sau, {\em Holomorphic functions on the symmetrized bidisk- Realization, interpolation and extension}, J. Funct. Anal. 274 (2018), 504-524.

 \bibitem{BK} M. Bhowmik and P. Kumar, {\em Bounded analytic functions on certain symmetrized domains}, preprint.

\bibitem{DKS} B. Krishna Das, P. Kumar and H. Sau, {\em Distinguished varieties and the Nevanlinna-Pick interpolation problem on the symmetrized bidisk}, arXiv:2104.12392.

\bibitem{DAS-SARKAR} B. Krishna Das and J. Sarkar {\em And\^o dilations, von Neumann inequality, and distinguished varieties}, J. Funct. Anal. 272 (2017), no. 5, 2114-2131.

\bibitem{Costara} C. Costara, {\em The symmetrized bidisc and Lempert’s theorem}, Bull. Lond. Math.
Soc. 36 (2004), 656–662.

\bibitem{DM}M. A. Dritschel and S. McCullough, {\em Test functions, kernels, realizations and interpolation}, Operator theory, structured matrices, and dilations, 153–179, Theta Ser. Adv. Math. 7, Theta, Bucharest, 2007. 

\bibitem{DMM}M. A. Dritschel, S. Marcantognini and  S. McCullough, {\em Interpolation in semigroupoid algebras},  J. Reine Angew. Math. 606 (2007), 1–40.

\bibitem{Alg} G. Fischer, {\em Plane algebraic curves},  Translated from the 1994 German original by Leslie Kay. Student Mathematical Library, 15. American Mathematical Society, Providence, RI, 2001. xvi+229 pp. ISBN: 0-8218-2122-9.

\bibitem{JK-JFA} M. Jury, G. Knese and S. McCullough, {\em Nevanlinna-Pick interpolation on distinguished
varieties in the bidisc}, J. Funct. Anal. 262 (2012), 3812–3838.

\bibitem{Knese-Trans} G. Knese, {\em Polynomials defining distinguished varieties}, Trans. Amer. Math, Soc. 362 (2010), 5635-5655.

\bibitem{K} \L{}. Kosi\'nski, {\em Three-point Nevanlinna-Pick problem in the polydisc}, Proc. Lond. Math. Soc. 111 (2015), 887–910. 

\bibitem{KZ_JGA} \L{}. Kosi\'nski and W. Zwonek, {\em Nevanlinna-Pick problem and uniqueness of left inverses in convex domains, symmetrized bidisc and tetrablock}, J. Geom. Anal. 26 (2016), no. 3, 1863–1890.

\bibitem{KZ-Tran} \L{}. Kosi\'nski and W. Zwonek, {\em Nevanlinna-Pick interpolation problem in the ball}, Trans. Amer. Math. Soc. 370 (2018), 3931-3947.
 
 \bibitem{KM} K. Maciaszek, {\em Geometry of uniqueness varieties for a three-point Pick problem in $\mathbb{D}^3$ }, arXiv:2204.06612.
 
 \bibitem{MSRZ} G. Misra, S. Shyam Roy and G. Zhang, {\em Reproducing kernel for a class of weighted Bergman spaces on the symmetrized polydisc}, Proc. Amer. Math. Soc. 141 (2013), no. 7, 2361–2370. 
 

\bibitem{Pal-Shalit} S. Pal and O. M. Shalit, {\em Spectral sets and distinguished varieties in the symmetrized bidisc}, J. Funct. Anal. 266 (2014), 5779-5800.

\bibitem{Paulsen} V. I. Paulsen and M. Raghupathi, {\em An introduction to the theory of reproducing kernel Hilbert spaces}, Cambridge Studies in Advanced Mathematics, 152. Cambridge University Press, Cambridge, 2016. x+182 pp. ISBN: 978-1-107-10409-9 46-02.

\bibitem{PICK} G. Pick, {\em \"{U}ber die Beschr\"{a}nkungen analytischer Funktionen, welche durch
vorgegebene Funktionswerte bewirkt werden}, Math. Ann. 77 (1916), 7–23.

\bibitem{Rudin} W. Rudin, {\em Function theory in polydiscs}, Benjamin, New York, 1969.


\bibitem{DS-CAOT} D. Scheinker, {\em A uniqueness theorem for bounded analytic functions on the polydisc}, Complex Anal. Oper. Theory 7 (2013), no. 5, 1429–1436.

\bibitem{DS-1} D. Scheinker, {\em Hilbert function spaces and the Nevanlinna-Pick problem on the polydisc}, J. Funct. Anal. 261 (2011), 2238-2249.

\bibitem{DS-2} D. Scheinker, {\em Hilbert function spaces and the Nevanlinna-Pick problem on the polydisc II}, J. Funct. Anal. 266 (2014), 355-367.

\bibitem{St} E. L. Stout, {\em Bounded extensions. The case of discs in polydiscs}, J. Anal. Math. 28 (1975), 239–
254.
 
\end{thebibliography}
\end{document}